\documentclass[12pt, twoside]{article}
\usepackage{amsmath,amsthm,amssymb}
\usepackage{times}
\usepackage{enumerate}
\usepackage{color}

\theoremstyle{definition}
\newtheorem{thm}{Theorem}[section]

\newtheorem{lem}[thm]{Lemma}
\newtheorem{defin}[thm]{Definition}

\numberwithin{equation}{section}

\newcommand{\subjclass}[1]{\bigskip\noindent\emph{2010 Mathematics Subject Classification:}\enspace#1}
\newcommand{\keywords}[1]{\noindent\emph{Keywords:}\enspace#1}

\begin{document}


\baselineskip=17pt


\title{Global weak solvability, continuous dependence on data, and large time growth of swelling moving interfaces}

\author{Kota Kumazaki\\
1-14 Bunkyo-cho, Nagasaki-city, Nagasaki, 851-8521, Japan\\
k.kumazaki@nagasaki-u.ac.jp\\
Adrian Muntean\\
Universitetsgatan 2, 651 88 Karlstad, Sweden\\
}

\date{}

\maketitle


\begin{abstract}
We prove a global existence result for weak solutions to a free boundary problem with flux boundary conditions describing swelling along a halfline. Additionally, we show that solutions are not only unique but also depend continuously on data and parameters. The key observation is  that the structure of our system of partial differential equations allows us to show that the moving {\em a priori} unknown interface never disappears. As main ingredients of the global existence proof, we rely on a local weak solvability result for our problem (as reported in \cite{KA}), uniform energy estimates of the solution, integral estimates on quantities defined at the free boundary, as well as a fine pointwise lower bound for the position of the moving boundary. Some of the estimates are time independent. They allow us to explore  the large-time behavior of the position of the moving boundary. The approach is specific to one-dimensional settings.



\end{abstract}

\section{Introduction}

Motivated by our attempt to understanding theoretically the formation of microscopic ice-lenses  growing inside unsaturated porous materials (like leftovers of hydration reactions in concrete, see particularly \cite{Setzer} but also \cite{Fin,Liu,Keener}) exposed to a cold weather, we proposed in our recent work \cite{KA} a one dimensional free boundary problem (FBP) to describe the swelling of water into a one-dimensional halfline. We have shown in \cite{KA} that the FBP is weakly solvable locally in time. In this paper, we prove that the  existence of our concept of solution can actually be extended globally in time.  We point out here that the word "halfline" refers  to our geometric description of a microscopic pore of the heterogeneous material. 
Two natural questions can be asked in this context:  
\begin{itemize}
\item[(Q1)] 
{\em How far the incoming (water) content can actually push the {\em a priori unknown} position of the moving boundary of swelling?}
\item[(Q2)]  {\em Can the microscopic pore eventually become empty once the swelling process has started?}
\end{itemize}
The question (Q1) motivates us to choose for a free boundary modeling strategy of the swelling process, while the answer to (Q2) turns out to be able to clarify a concept of global existence of solutions to the evolution problem behind (Q1).

It is worth noting that from the mathematical standpoint, our free boundary problem resembles remotely the classical one phase Stefan problem and its variations for handling unsaturated flow through capillary fringes, superheating, phase transitions, or evaporation; compare, for instance, \cite{Evaporation, disso-precipi, Helmig, Xie} and references cited therein. Our contribution is in line with the existing mathematical modeling and analysis work of swelling by Fasano and collaborators (see \cite{FMP,FM}, e.g.),  as well as of others including  e.g. Refs. \cite{Calderer,Otani,Z}.

We now describe the setting of our equations to be viewed in dimensionless form. For $a, T\in (0,+\infty)$, $[0,T]$ is the timespan and $[a,+\infty)$ represents one pore (pocket, halfline). The variables are $t\in [0,T]$ and $z\in [a,+\infty)$. 
The boundary $z=a$ denotes the edge of the pore in contact with wetness.  The interval $[a, s(t)]$ indicates the region of diffusion of the  water content $u(t)$, where $s(t)$ is the moving interface of the water region.  The function $u(t)$ acts in the non-cylindrical region $Q_s(T)$ defined by
\begin{align*}
& Q_s(T):=\{(t, z) | 0<t<T, \ a<z<s(t) \}. 
\end{align*}

Our free boundary problem, which we denote by $(\mbox{P})_{u_0, s_0, b}$, is as follows: Find the pair $(u(t,z), s(t))$ satisfying the following set of equations as well as initial and boundary conditions, viz. 
\begin{align}
& u_t-ku_{zz}=0 \mbox{ for }(t, z)\in Q_s(T), \label{p1-1}\\
& -ku_z(t, a)=\beta(b(t)-Hu(t, a)) \mbox{ for }t\in(0, T), \label{p1-2}\\
& -ku_z(t, s(t))=u(t, s(t))s_t(t) \mbox{ for }t\in (0, T), \label{p1-3}\\
& s_t(t)=a_0(u(t, s(t))-\varphi(s(t))) \mbox{ for }t\in (0, T), \label{p1-4}\\
& s(0)=s_0, u(0, z)=u_0(z) \mbox{ for }z \in [a, s_0]. \label{p1-5}
\end{align}
Here $k>0$ is a diffusion constant, $\beta$ is a given adsorption function on $\mathbb{R}$ that is equal to 0 for negative input and takes a positive value for positive input, $b$ is a given moisture threshold function on $[0, T]$,  $H$ and $a_0$ are further given (positive) constants, $\varphi$ is our breaking function defined on $\mathbb{R}$, while  $s_0$ and $u_0$ are the initial data. 

From the physical point of view, (1.1) is the diffusion equation displacing the moisture content $u$ in the unknown region $[a, s]$; the boundary condition (1.2),  imposed at $z=a$,  implies that the moisture content $b$ inflows if $b$ is present  at $z=a$ in a larger amount than $u$ while the moisture flow stops otherwise. The boundary condition (1.3) at $z=s(t)$ describes the mass conservation at the moving boundary. The free boundary condition (1.4) is of kinetic type, i.e. it is an explicit description of the speed of the moving boundary having two competing components -- the mechanisms of spreading $a_0(u(t, s(t))$ and of  breaking $-a_0\varphi(s(t)))$.  For the problem 
$(\mbox{P})_{u_0, s_0, b}$ to make sense (both mathematically and physically), we assume suitable conditions for $\beta$, $\varphi$, $s_0$ and $u_0$ so that one can speak about unique solutions $(s, u)$ on $[0, T_0]$ for some $0<T_0\leq T$. 

The aim of this paper is to construct a global-in-time  solution of $(\mbox{P})_{u_0, s_0, b}$.
For this purpose, we obtain uniform estimates for solutions with respect to time $t$ (Section 3). Based on such uniform estimates, we can think of extending solutions in time. Let $[0, T^*)$ be the maximal time interval of the existence of a solution $(s, u)$ to our FBP. 
For such extensions of local solutions to $(\mbox{P})_{u_0, s_0, b}$ to take place, we need to consider $s(T^*)$ and $u(T^*)$ as initial data. 
The main ingredient in constructing a global solution to $(\mbox{P})_{u_0, s_0, b}$ lies on proving that $s(T^*)>a$, i.e. singularities due to the use of the employed Landau-type fixed domain transformation  are not captured in the solution. 
For instance,  consider for a moment $(\mbox{P})_{u(T^*), s(T^*), b}$, which is in fact  the problem (\ref{p1-1})--(\ref{p1-4}) with $s(T^*)$ and $u(T^*)$ as initial data. If we have $s(T^*)=a$, then there is no space for swelling at an initial time, and therefore, we can not reserve a domain so that  the solution of $(\mbox{P})_{u(T^*), s(T^*), b}$ develops. In this regard, by using the structure of equations and boundary conditions well, we succeed to prove that the free boundary has a pointwise lower bound which is strictly grater than $a$. This result indicates that the position of the free boundary is always grater than $a$ as time elapses, guaranteeing that $s(T^*)>a$. This is an essential ingredient of our proof.


The paper is organized as follows: In Section \ref{na}, we state the used notation and assumptions. We present our results in Section \ref{results}: our main theorem concerns the existence and uniqueness of a globally-in-time solution to the FBP modeling swelling along a halfline. We complement this with a result on the stability of solutions with respect to the initial position of the free boundary and initial moisture concentration as well as to the moisture content $b(t)$ entering the production term by Henry's law (\ref{p1-2}). Finally, to get a better feeling on the quality of the model, we investigate in Section \ref{large} the large time behavior of the free boundary position. Essentially the moving interfaces grows indefinetely unless the production term by Henry's law has a certain decay in time. These additional estimates should be seen as an invitation for deeper investigations of the large time behavior of solutions to our problem, particularly about deriving quantitative convergence rates to the expected steady states. As working strategy,   we first show that the position of the free boundary admits a lower bound that is strictly greater than $a$; the key ingredient is Lemma \ref{lem2}. This gives a partial answer to question (Q2). Next, we obtain useful uniform energy estimates of the solution by multiplying the difference quotient of a solution to the governing equation and combining with integral estimates on quantities defined on the free boundary. We then prove our main theorem by extending local solutions  $(\mbox{P})_{u_0, s_0, b}$ up to global ones addressing this way the question (Q1). We reuse energy estimates as well as integral estimates on fluxes at the free boundary to complement the analysis with a continuous dependence result of the solution of the moving boundary problem with respect to the given data and model parameters.





\section{Notation, assumptions and concept of solution}
\label{na}

In this paper, we use the following basic notations. We denote by $| \cdot |_X$ the norm for a Banach space $X$. The norm and the inner product of a Hilbert space $H$ are
 denoted by $|\cdot |_H$ and $( \cdot, \cdot )_H$, respectively. Particularly, for an open interval $\Omega \subset {\mathbb R}$, we use the standard notation of the usual Hilbert spaces $L^2(\Omega)$, $H^1(\Omega)$ and $H^2(\Omega)$.

Throughout this paper, we assume the following restrictions on the involved parameters and functions:

(A1) $a$, $a_0$, $H$, $k$ and $T$ are positive constants. 

(A2) $b\in W^{1,2}(0, T)\cap L^{\infty}(0, T)$ with $0<b_*\leq b\leq b^*$ on $(0, T)$, where $b_*$ and $b^*$ are positive constants.

(A3) $\beta\in C^1(\mathbb{R)}\cap W^{1, \infty}(\mathbb{R})$ such that $\beta=0$ on $(-\infty, 0]$, and there exists $r_{\beta}>0$ such that $\beta'>0$ on $(0, r_{\beta})$ and $\beta = k_0$ on $[r_{\beta}, +\infty)$, where $k_0$ is a positive constant. 

(A4) $\varphi \in C^1(\mathbb{R})\cap W^{1, \infty}(\mathbb{R})$ such that $\varphi=0$ on $(-\infty, 0]$, and there exists $r_{\varphi}>0$ such that $\varphi' >0$ on $(0, r_{\varphi})$ and $\varphi = c_0$ on $[r_{\varphi}, +\infty)$, where $c_0$ is a positive constant satisfying 
\begin{align}
\label{c0-1}
0<c_0<\mbox{min}\{2\varphi(a), b^*H^{-1}\}. 
\end{align}
Also, we put $c_{\varphi}=\mbox{sup}_{r\in \mathbb{R}}\varphi'(r)$. 

(A5) $a<s_0<r_{\varphi}$, where $r_{\varphi}$ is the same as in (A4), and $u_0\in H^1(a, s_0)$ such that $\varphi(a)\leq u_0\leq b^*H^{-1}$ on $[a, s_0]$.

For $T>0$, let $s$ be a function on $[0, T]$ and $u$ be a function on $Q_s(T):=\{(t, z) | 0\leq t\leq T, a<s(t)\}$.

Next, we define our concept of solution to (P)$_{u_0, s_0, b}$ on $[0,T]$ in the following way:

\begin{defin}
\label{def}
 We call that pair $(s, u)$ a solution to (P)$_{u_0, s_0, b}$ on $[0, T]$ if the following conditions (S1)-(S6) hold:

(S1) $s$, $s_t\in L^{\infty}(0, T)$, $a<s$ on $[0, T]$, $u\in L^{\infty}(Q_s(T))$, $u_t$, $u_{zz}\in L^2(Q_s(T))$ and $t \in [0, T] \to |u_z(t,  \cdot)|_{L^2(a, s(t))}$ is bounded;

(S2) $u_t-ku_{zz}=0$ a.e. on $Q_s(T)$;

(S3) $-ku_z(t, a)=\beta(b(t)-Hu(t, a))$ for a.e. $t\in (0, T)$;

(S4) $-ku_z(t, s(t))=u(t, s(t))s_t(t)$ for a.e. $t\in (0, T)$;

(S5) $s_t(t)=a_0(u(t, s(t))-\varphi(s(t)))$ for a.e. $t\in (0, T)$;

(S6) $s(0)=s_0$ and $u(0, z)=u_0(z)$ for $z \in [a, s_0]$.
\end{defin}

\section{Results}\label{results}

In this part, we report on our results concerning the solvability of the proposed moving boundary problem. We are building upon our previous understanding of the local-in-time existence of weak solutions; see \cite{KA}, and present essentially new results.
The first main result of this paper is concerned with the existence and uniqueness of a time global solution in the sense of Definition \ref{def} to the problem (P)$_{u_0, s_0, b}$.
\begin{thm}
\label{th2}
Let $T>0$. If (A1)--(A5) hold, then (P)$_{u_0, s_0, b}$ has a unique solution $(s, u)$ on $[0, T]$ satisfying 
$\varphi(a)\leq u \leq b^*H^{-1}$ on $Q_s(T)$. 
\end{thm}

To prove Theorem \ref{th2}, we transform (P)$_{u_0, s_0, b}$, initially posed in a non-cylindrical domain, to a cylindrical domain. Let $T>0$. For given $s\in W^{1,2}(0, T)$ with $a<s(t)$ on $[0, T]$, we introduce the following new function obtained by the indicated change of variables:
\begin{align}
\label{p2-0}
& \tilde{u}(t, y)=u(t, (1-y)a+ys(t)) \mbox{ for } (t, y)\in Q(T):=(0, T)\times (0, 1). 
\end{align}
By using the function $\tilde{u}$, we consider now the following problem $(\mbox{P})_{\tilde{u}_0, s_0, b}$:
\begin{align}
& \tilde{u}_t(t, y)-\frac{k}{(s(t)-a)^2}\tilde{u}_{yy}(t, y)=\frac{ys_t(t)}{s(t)-a}\tilde{u}_y(t, y) \mbox{ for }(t, y)\in Q(T), \label{p2-1}\\
& -\frac{k}{s(t)-a}\tilde{u}_y(t, 0)=\beta(b(t)-H\tilde{u}(t, 0)) \mbox{ for }t\in(0, T), \label{p2-2}\\
& -\frac{k}{s(t)-a}\tilde{u}_y(t, 1)=\tilde{u}(t, 1)s_t(t) \mbox{ for }t\in (0, T),  \label{p2-3}\\
& s_t(t)=a_0(\tilde{u}(t, 1)-\varphi(s(t))) \mbox{ for }t\in (0, T), \label{p2-4}\\
& s(0)=s_0, \label{p2-5}\\
& \tilde{u}(0, y)=u_0((1-y)a+y s(0))(:=\tilde{u}_0(y)) \mbox{ for }y \in [0, 1]. \label{p2-6}
\end{align}

The definition of a solution of $(\mbox{P})_{\tilde{u}_0, s_0, b}$ is as follows:
\begin{defin}
\label{def1}
For $T>0$, let $s$ be functions on $[0, T]$ and $\tilde{u}$ be a function on $Q(T)$, respectively. We call that a pair $(s, \tilde{u})$ is a solution of (P)$_{\tilde{u}_0, s_0, b}$ on $[0, T]$ if the conditions (S'1)-(S'2) hold:

(S'1) $s$, $s_t\in L^{\infty}(0, T)$, $a<s$ on $[0, T]$, $\tilde{u}\in W^{1,2}(Q(T))\cap L^{\infty}(0, T; H^1(0, 1))\cap L^2(0, T;H^2(0, 1))\cap L^{\infty}(Q(T))$.

(S'2) (\ref{p2-1})--(\ref{p2-6}) hold.
\end{defin}

Theorem \ref{th2} is a direct consequence of the following result concerning  the existence and uniqueness of a time global solution of problem (P)$_{\tilde{u}_0, s_0, b}$ transformed to a cylindrical domain and using the change of variable 
\begin{align}
\label{p2-7}
 u(t, z)=\tilde{u}\left(t, \frac{z-a}{s(t)-a}\right) \mbox{ for } (t, z)\in [0, T]\times [a, s(t)].
\end{align}
\begin{thm}
\label{th3}
Let $T>0$. If (A1)--(A5) hold, (P)$_{\tilde{u}_0, s_0, b}$ has a unique solution $(s, \tilde{u})$ on $[0, T]$ 
satisfying $\varphi(a)\leq \tilde{u} \leq b^*H^{-1}$ on $Q(T)$. 
\end{thm}

The second main result is concerned with the continuous dependence estimate between the solution $(s, \tilde{u})$ of (P)$_{\tilde{u}_0, s_0, b}$ and the given data $\tilde{u}_0$, $s_0$ and $b$:
\begin{thm}
\label{b-4}
For $T>0$ and $i=1$, 2, let $b_i\in W^{1, 2}(0, T)$ be a function satisfying (A2), $a<s_{0i}<r_{\varphi}$ and $u_{0i}\in H^1(a, s_{0i})$ with $\varphi(a)\leq u_{0i}\leq b^*H^{-1}$ on $[a, s_{0i}]$ 
and let $(s_i, u_i)$ a solution of (P)$_{u_{0i}, s_{0i}, b_i}$ on $[0, T]$ satisfying $\varphi(a)\leq u_i \leq b^*H^{-1}$ on $Q_{s_i}(T)$. 
Then, there exists a positive constant $M_T$ depending on $T$ such that 
\begin{align}
\label{eqb-22}
& |s_1(t)-s_2(t)|^2+|\tilde{u}_1(t)-\tilde{u}_2(t)|^2_{L^2(0, 1)} + \int_0^t |\tilde{u}_{1y}(\tau)-\tilde{u}_{2y}(\tau)|^2_{L^2(0, 1)} d\tau \nonumber \\
\leq & M_T \left( |s_{01}-s_{02}|^2+|\tilde{u}_{01}-\tilde{u}_{02}|^2_{L^2(0, 1)} + \int_0^t |b_1(\tau)-b_2(\tau)|^2 d\tau \right) \mbox{ for }t\in [0, T],
\end{align}
where $\tilde{u}_i$ and $\tilde{u}_{0i}$ are given by the similar way to  (\ref{p2-0}) and (\ref{p2-6}), respectively.
\end{thm}

In Section 5,  we discuss the large time behavior of a solution to (P)$_{u_0, s_0, b}$ as $t\to \infty$. In fact, we are able to prove the following dichotomy: It either  holds 
$$ (i)  \lim_{t \to \infty} \int_0^t \beta(b(\tau)-Hu(\tau, a))d\tau=\infty \mbox{ and } \lim_{t\to \infty}s(t)=\infty,$$
or 
$$(ii) \lim_{t \to \infty} \int_0^t \beta(b(\tau)-Hu(\tau, a))d\tau \leq  \frac{b^*}{H}(r_{\varphi}-r_*) \mbox{ and } \limsup_{t \to \infty}s(t)\leq a+ \frac{1}{\varphi(a)} \frac{b^*}{H}(s_0-a+r_{\varphi}-r_*). $$
In other words, this result emphasizes that the elongation-by-swelling of the one-dimensional pore stays finite as $t\to\infty$ only if the production term by Henry's law decays in time.  For a precise statement of the result, the reader is referred to Theorem \ref{b-2} and  also to Section \ref{bb} for the corresponding proof. The proofs of Theorems \ref{th3} and \ref{b-4} are given in Section 3 and 4.

\section{Uniform estimates}

In this section, we prove a number of uniform estimates for the solution. Throughout this section, we always assume (A1)--(A5) (without stating it each time when used).  
First, we show a uniform lower estimate of the free boundary $s$; see (\ref{KEY}). This is a key ingredient ensuring the global existence of solutions to our problem.  
\begin{lem}
\label{lem2}
Let $(s, u)$ be a solution of (P)$_{u_0, s_0, b}$ on $[0, T]$ satisfying $\varphi(a)\leq u \leq b^*H^{-1}$ on $[a, s(t)]$ for $t\in [0, T]$. Then, there exists a positive constant $s_*$ such that $a<s_* <s_0$ and 
\begin{align}
\label{KEY}
s(t) \geq s_* \mbox{ for }t\in [0, T].
\end{align}
\end{lem}
\begin{proof}
First, by integrating (\ref{p1-1}) over $[a, s(t)]$ for $t\in [0, T]$, we have that 
\begin{align}
\label{eq3-1}
& \frac{d}{dt}\int_a^{s(t)} u(t) dz - s_t(t)u(t, s(t))-ku_z(t, s(t))+ku_z(t, a)=0 \mbox{ for a.e. }t\in [0, T]. 
\end{align}
Using the boundary conditions (\ref{p1-2}) and (\ref{p1-3}), we observe that 
\begin{align}
\label{eq3-2}
& \frac{d}{dt}\int_a^{s(t)}u(t) dz - s_t(t)u(t, s(t)) \nonumber \\
& + u(t, s(t))s_t(t) - \beta(b(t)-Hu(t, a)) =0 \mbox{ for a.e. }t\in [0, T]. 
\end{align}
The second and third terms are canceled out, and from the positivity of $\beta$ in (A3), we infer from (\ref{eq3-2}) that 
\begin{align}
\label{eq3-3}
\frac{d}{dt}\int_a^{s(t)}u(t) dz \geq 0 \mbox{ for a.e. }t\in [0, T]. 
\end{align}
Therefore, by integrating (\ref{eq3-3}) over $[0, t]$ for $t\in [0, T]$ it follows that 
\begin{align}
\label{eq3-4}
\int_a^{s(t)} u(t) dz \geq \int_a^{s_0} u(0) dz.
\end{align}
Since  $\varphi(a)\leq u \leq b^*H^{-1}$ on $[a, s(t)]$ for $t\in [0, T]$ we obtain from (\ref{eq3-4}) that 
\begin{align}
\label{eq3-5}
s(t) \geq a + \frac{H}{b^*}\varphi(a)(s_0-a).
\end{align}
The constant in the right hand side of (\ref{eq3-5}) is the desired one. In fact, by $a<s_0$ it is easy to see that $a<a + \frac{H}{b^*}\varphi(a)(s_0-a)$. 
Also, by $\varphi(a) < c_0<b^*H^{-1}$ in (\ref{c0-1}) we see that 
\begin{align}
a + \frac{H}{b^*} \varphi(a)(s_0-a) -s_0 &=(s_0-a) \left(\frac{H}{b^*} \varphi(a)-1\right)<0. \nonumber 
\end{align}
Thus, Lemma \ref{lem2} is proved. 
\end{proof}

Next, we prove uniform estimates of the solution $u$ of (P)$_{u_0, s_0, b}$. 
\begin{lem}
\label{lem3-1}
Let $(s, u)$ be a solution of (P)$_{u_0, s_0, b}$ on $[0, T]$ satisfying $\varphi(a)\leq u \leq b^*H^{-1}$ on $[a, s(t)]$ for $t\in [0, T]$. Then, there exists a positive constant $C(T)$ which depends on $T$ such that 
\begin{align}
\label{eq4-2}
& \int_{0}^t|u_t(\tau)|^2_{L^2(a, s(\tau))}d\tau + |u_z(t)|^2_{L^2(a, s(t))} \leq C(T) \mbox{ for all }t\in (0, T). 
\end{align}
\end{lem}
\begin{proof}
Let $(s, u)$ be a solution of (P)$_{u_0, s_0, b}$ on $[0, T]$. Then, by the change of variables (\ref{p2-0}) we see that $(s, \tilde{u}$) is a solution of (P)$_{\tilde{u}_0, s_0, b}$ on $[0, T]$ in the sense of Definition \ref{def1}. Now, we put $v_h(t)=\frac{\tilde{u}(t)-\tilde{u}(t-h)}{h}$ for $h>0$ and $u(t)=u(0)=u_0$ for $t<0$. 
By (\ref{p2-1}),  we have that 
\begin{align}
\label{eq4-8}
& \int_0^1 \tilde{u}_t(t)(s(t)-a)v_h(t) dy -k\int_0^1\frac{1}{s(t)-a}\tilde{u}_{yy}(t)v_h(t) dy \nonumber \\
& =\int_0^1 \frac{ys_t(t)\tilde{u}_y(t)}{s(t)-a} (s(t)-a) v_h(t) dy \quad \mbox{ for }t\in [0, T]. 
\end{align}
Then, using the boundary conditions (\ref{p2-2}) and (\ref{p2-3}), we observe that 
\begin{align}
\label{eq4-9}
& -\int_0^1\frac{k}{s(t)-a}\tilde{u}_{yy}(t) \frac{\tilde{u}(t)-\tilde{u}(t-h)}{h} dy \nonumber \\
= & \tilde{u}(t, 1)(a_0(\tilde{u}(t, 1)-\varphi(s(t))))v_h(t, 1) \nonumber \\
& -\beta(b(t)-H\tilde{u}(t, 0)) v_h(t, 0) + \int_0^1 \frac{k}{s(t)-a}\tilde{u}_y(t)v_{hy}(t) dy, 
\end{align}
and 
\begin{align}
\label{eq4-10}
& \int_0^1 \frac{k}{s(t)-a}\tilde{u}_y(t)v_{hy}(t) dy \nonumber \\
\geq & \frac{k}{2h}\int_0^1 \frac{1}{s(t)-a}(|\tilde{u}_y(t)|^2-|\tilde{u}_y(t-h)|^2)dy \nonumber \\
= & \frac{k}{2h}\biggl[ \int_a^{s(t)}|u_z(t)|^2 dz - \int_a^{s(t-h)}\frac{s(t-h)-a}{s(t)-a}|u_z(t-h)|^2 dz\biggr] \nonumber \\
= & \frac{k}{2h}\biggl[ \int_a^{s(t)}|u_z(t)|^2 dz - \int_a^{s(t-h)}|u_z(t-h)|^2 dz + \int_a^{s(t-h)}\frac{s(t)-s(t-h)}{s(t)-a}|u_z(t-h)|^2 dz\biggr].
\end{align}
Here, we introduce the function $\Phi(s(t), r)=a_0\left(\frac{r^3}{3}-\varphi(s(t))\frac{r^2}{2}\right)$ for $t\in[0, T]$ and $r\in \mathbb{R}$. 
Then, $\frac{\partial}{\partial r}\Phi(s(t), r)=a_0(r^2-\varphi(s(t))r)$ and $\frac{\partial^2}{\partial r^2}\Phi(s(t), r)=a_0(2r-\varphi(s(t)))$ for $t\in [0, T]$. 
In particular, by (A4), it holds that for $r\geq \varphi(a)$, 
\begin{align}
\label{eq4-11}
\frac{\partial^2}{\partial r^2}\Phi(s(t), r) &=a_0(2r-\varphi(s(t))) \nonumber \\
                                   &\geq a_0(2\varphi(a)-\varphi(s(t))) \nonumber \\
                                   & >0 \quad \mbox{ for }t\in [0, T].
\end{align}
Since $\tilde{u}(t, 1)\geq \varphi(a)$ for $t\in [0, T]$, (\ref{eq4-11}) implies, for $t\in [0, T]$,  that $\Phi(s(t), \tilde{u}(t, 1))$ is convex with respect to the second component. From this, we see that the following inequality holds.
\begin{align}
\label{eq4-12}
\tilde{u}(t, 1)(a_0(\tilde{u}(t, 1)-\varphi(s(t)))v_h(t, 1)\geq \frac{\Phi(s(t), \tilde{u}(t, 1))-\Phi(s(t), \tilde{u}(t-h, 1))}{h} 
\mbox{ for }t\in [0, T]. 
\end{align}
Next, we also introduce $\hat{\beta}(b(t), r)=\int_0^r -\beta(b(t)-H\tau)d\tau$ for $t\in [0, T]$ and $r\in \mathbb{R}$. Then, by (A3) it is easy to see that $\frac{\partial^2}{\partial r^2}\hat{\beta}(b(t), r)=\beta'(b(t)-Hr)H \geq 0$. Hence, for $t\in [0, T]$, $\hat{\beta}(b(t), \tilde{u}(t, 0))$ is convex with respect to the second component so that we can see that the following inequality holds. 
\begin{align}
\label{eq4-13}
-\beta(b(t)-H\tilde{u}(t, 0))v_h(t, 0) \geq \frac{\hat{\beta}(b(t), \tilde{u}(t, 0))-\hat{\beta}(b(t), \tilde{u}(t-h, 0))}{h} 
\mbox{ for }t\in [0, T].
\end{align}
Combining (\ref{eq4-9})-(\ref{eq4-13}) as well as  (\ref{eq4-8}) yields the inequality  
\begin{align}
\label{eq4-14}
& \int_0^1\tilde{u}_t(t) (s(t)-a)v_h(t)dy \nonumber \\
& + \frac{k}{2h}\biggl[ \int_a^{s(t)}|u_z(t)|^2 dz - \int_a^{s(t-h)}|u_z(t-h)|^2 dz + \int_a^{s(t-h)}\frac{s(t)-s(t-h)}{s(t)-a}|u_z(t-h)|^2 dz\biggr] \nonumber \\
& + \frac{\Phi(s(t), \tilde{u}(t, 1))-\Phi(s(t), \tilde{u}(t-h, 1))}{h} + \frac{\hat{\beta}(b(t), \tilde{u}(t, 0))-\hat{\beta}(b(t), \tilde{u}(t-h, 0))}{h} \nonumber \\
\leq & \int_0^1 ys_t(t)\tilde{u}_y(t) v_h(t) dy \mbox{ for }t\in [0, T].
\end{align}
Now, we integrate (\ref{eq4-14}) over $[0, t_1]$ for $t_1\in (0 ,T]$ and take the limit as $h\to 0$. Then, by the change of the variables the first term of the left hand side of (\ref{eq4-14}) is as follows:
\begin{align}
\label{eq4-15}
& \lim_{h\to 0}\int_0^{t_1} \int_0^1 \tilde{u}_t(t) (s(t)-a)v_h(t) dy dt = \int_0^{t_1}\int_0^1 |\tilde{u}_t(t)|^2(s(t)-a) dy dt \nonumber \\
=& \int_0^{t_1}\int_a^{s(t)} \biggl(|u_t(t)|^2 +2u_t(t)u_z(t)\frac{z-a}{s(t)-a}s_t(t) + \left(u_z(t)\frac{z-a}{s(t)-a}s_t(t)\right)^2 \biggr) dzdt.
\end{align}
Argueing similarly as in in Lemma 3.4 of \cite{KA}, we can prove that the function $t\to \int_a^{s(t)}|u_z(t)|^2dz$ is absolutely continuous on $[0, T]$. Then, the second and third terms of  the left-hand side of (\ref{eq4-14}) can be dealt with as 
\begin{align}
\label{eq4-16}
& \lim_{h\to0} \frac{k}{2h}\int_0^{t_1}\biggl(\int_a^{s(t)}|u_z(t)|^2dz-\int_a^{s(t-h)}|u_z(t-h)|^2dz\biggr)d\tau \nonumber \\
= & \frac{k}{2}\biggl(\int_a^{s(t_1)}|u_z(t_1)|^2dz - \int_a^{s_0}|u_z(0)|^2dz\biggr), 
\end{align}
and 
\begin{align}
\label{eq4-17}
&\lim_{h\to 0}\frac{k}{2}  \int_0^{t_1}\int_a^{s(t-h)} \frac{1}{s(t)-a} \frac{s(t)-s(t-h)}{h}|u_z(t-h)|^2 dzdt \nonumber \\
=& \frac{k}{2} \int_0^{t_1}\int_a^{s(t)} \frac{s_t(t)}{s(t)-a}|u_z(t)|^2dzdt.
\end{align}
Moreover, since $\tilde{u}$ is continuous on $\overline{Q(T)}$ we have that 
\begin{align}
\label{eq4-18}
& \lim_{h\to 0}\frac{1}{h}\int_0^{t_1}\biggl(\Phi(s(t), \tilde{u}(t, 1))-\Phi(s(t), \tilde{u}(t-h, 1))\biggr)dt \nonumber \\
& = \lim_{h\to 0} \biggl(\frac{1}{h}\int_{t_1-h}^{t_1} \Phi(s(t), \tilde{u}(t, 1))dt\biggr) -\Phi(s(0), \tilde{u}_0(1))
+\frac{a_0}{2}\int_0^{t_1}\varphi'(s(t))s_t(t)|\tilde{u}(t, 1)|^2 dt \nonumber \\
& = \Phi(s(t_1), \tilde{u}(t_1, 1))-\Phi(s(0), \tilde{u}_0(1))+\frac{a_0}{2}\int_0^{t_1}|\tilde{u}(t, 1)|^2\varphi'(s(t))s_t(t) dt, 
\end{align}
and similarly to the derivation of (\ref{eq4-18}) 
\begin{align}
\label{eq4-19}
& \lim_{h\to 0}\frac{1}{h}\int_0^{t_1}\hat{\beta}(b(t), \tilde{u}(t, 0))-\hat{\beta}(b(t), \tilde{u}(t-h, 0)) dt \nonumber \\ 
= & \hat{\beta}(b(t_1), \tilde{u}(t_1, 0)) - \hat{\beta}(b(0), \tilde{u}_0(0)) \nonumber \\ 
& + \lim_{h\to 0}\biggl(-\frac{1}{h}\int_0^{t_1}\biggl[\hat{\beta}(b(t), \tilde{u}(t-h, 0))-\hat{\beta}(b(t-h), \tilde{u}(t-h, 0))\biggr]dt\biggr). 
\end{align}
For the last term of the right hand side of (\ref{eq4-19}) we observe that 
\begin{align}
\label{eq4-20}
& \lim_{h\to 0}\biggl(-\frac{1}{h}\int_0^{t_1}\biggl[\hat{\beta}(b(t), \tilde{u}(t-h, 0))-\hat{\beta}(b(t-h), \tilde{u}(t-h, 0))\biggr]dt\biggr) \nonumber \\
\geq & \lim_{h\to 0}\biggl(-\frac{1}{h}\int_0^{t_1} |\beta'|_{L^{\infty}(\mathbb{R})} |b(t)-b(t-h)||\tilde{u}(t-h, 0)|dt \biggr) \nonumber \\
\geq & \lim_{h\to 0}\biggl(-\frac{|\beta'|_{L^{\infty}(\mathbb{R})}}{h}\int_0^{t_1}\biggl(\int_{t-h}^t |b_t(\tau)|d\tau\biggr)|\tilde{u}(t-h, 0)|dt \biggr) \nonumber \\
\geq & -|\beta'|_{L^{\infty}(\mathbb{R})}\frac{b^*}{H} \int_0^{t_1}|b_t(t)| dt.
\end{align}
From the estimates (\ref{eq4-15})-(\ref{eq4-20}), we obtain from (\ref{eq4-14}) that 
\begin{align}
\label{eq4-21}
& \int_0^{t_1}\int_a^{s(t)} \biggl(|u_t(t)|^2 +2u_t(t)u_z(t)\frac{z-a}{s(t)-a}s_t(t) + \left(u_z(t)\frac{z-a}{s(t)-a}s_t(t)\right)^2 \biggr) dzdt  \nonumber \\
& + \frac{k}{2}\int_a^{s(t_1)}|u_z(t_1)|^2dz -\frac{k}{2}\int_a^{s_0}|u_z(0)|^2dz + \frac{k}{2}\int_0^{t_1}\int_a^{s(t)} \frac{s_t(t)}{s(t)-a}|u_z(t)|^2dzdt \nonumber \\
& + \Phi(s(t_1), \tilde{u}(t_1, 1)) -  \Phi(s(0), \tilde{u}_0(1)) + \frac{a_0}{2}\int_0^{t_1}|\tilde{u}(t, 1)|^2\varphi'(s(t))s_t(t) dt \nonumber \\
& + \hat{\beta}(b(t_1), \tilde{u}(t_1, 0))-\hat{\beta}(b(0), \tilde{u}_0(0)) - |\beta'|_{L^{\infty}(\mathbb{R})}\frac{b^*}{H}\int_0^{t_1}|b_t(t)|dt  \nonumber \\
\leq & \int_0^{t_1} \int_a^{s(t)} \biggl (u_t(t)u_z(t)\frac{z-a}{s(t)-a}s_t(t) + \left(u_z(t)\frac{z-a}{s(t)-a}s_t(t)\right)^2 \biggr) dzdt \ \mbox{ for }t_1\in [0, T].
\end{align}
Therefore, we have 
\begin{align}
\label{eq4-21}
& \int_0^{t_1}\int_a^{s(t)} |u_t(t)|^2 dzdt  \nonumber \\
& + \frac{k}{2}\int_a^{s(t_1)}|u_z(t_1)|^2dz -\frac{k}{2}\int_a^{s_0}|u_z(0)|^2dz + \frac{k}{2}\int_0^{t_1}\int_a^{s(t)} \frac{s_t(t)}{s(t)-a}|u_z(t)|^2dzdt \nonumber \\
& + \Phi(s(t_1), \tilde{u}(t_1, 1)) -  \Phi(s(0), \tilde{u}_0(1)) + \frac{a_0}{2}\int_0^{t_1}|\tilde{u}(t, 1)|^2\varphi'(s(t))s_t(t) dt \nonumber \\
& + \hat{\beta}(b(t_1), \tilde{u}(t_1, 0))-\hat{\beta}(b(0), \tilde{u}_0(0)) - |\beta'|_{L^{\infty}(\mathbb{R})}\frac{b^*}{H} \int_0^{t_1}|b_t(t)| dt \nonumber \\
\leq & \int_0^{t_1} \int_a^{s(t)} -u_t(t)u_z(t)\frac{z-a}{s(t)-a}s_t(t) dzdt \ \mbox{ for }t_1\in [0, T].
\end{align}
Here, we note that 
\begin{align}
\label{eq4-22}
& -\int_0^{t_1} \int_a^{s(t)} u_t(t)u_z(t)\frac{z-a}{s(t)-a}s_t(t) dzdt \nonumber \\
= & -\int_0^{t_1} \int_a^{s(t)} ku_{zz}(t) u_z(t) \frac{z-a}{s(t)-a}s_t(t) dzdt \nonumber \\
= & -\int_0^{t_1} \int_a^{s(t)} \frac{k}{2}\biggl(\frac{\partial}{\partial z}|u_z(t)|^2\biggr) \frac{z-a}{s(t)-a}s_t(t) dzdt \nonumber \\
= & -\int_0^{t_1} \frac{k}{2}|u_z(t, s(t))|^2s_t(t) dt + \frac{k}{2}\int_0^{t_1}\int_a^{s(t)}\frac{s_t(t)}{s(t)-a}|u_z(t)|^2dzdt. 
\end{align}
Hence, by applying (\ref{eq4-22}) to (\ref{eq4-21}) we obtain that 
\begin{align}
\label{eq4-23}
& \int_0^{t_1}\int_a^{s(t)} |u_t(t)|^2 dzdt \nonumber \\ 
& + \frac{k}{2}\int_a^{s(t_1)}|u_z(t_1)|^2dz - \frac{k}{2}\int_a^{s_0}|u_z(0)|^2dz + \frac{k}{2}\int_0^{t_1}\int_a^{s(t)} \frac{s_t(t)}{s(t)-a}|u_z(t)|^2dzdt \nonumber \\
& + \Phi(s(t_1), \tilde{u}(t_1, 1)) -  \Phi(s(0), \tilde{u}_0(1)) + \frac{a_0}{2}\int_0^{t_1}|\tilde{u}(t, 1)|^2\varphi'(s(t))s_t(t) dt \nonumber \\
& + \hat{\beta}(b(t_1), \tilde{u}(t_1, 0))-\hat{\beta}(b(0), \tilde{u}_0(0)) - |\beta'|_{L^{\infty}(\mathbb{R})}\frac{b^*}{H} \int_0^{t_1}|b_t(t)| dt \nonumber \\ 
\leq & -\int_0^{t_1} \frac{k}{2}|u_z(t, s(t))|^2s_t(t) dt + \frac{k}{2}\int_0^{t_1}\int_a^{s(t)}\frac{s_t(t)}{s(t)-a}|u_z(t)|^2dzdt \ \mbox{ for }t_1\in [0, T].
\end{align}
The forth term in the left-hand side and the last term in the right-hand side are canceled out and by the definition of $\hat{\beta}$, $-\hat{\beta}(b(0), \tilde{u_0}(0))$ is positive. 
Also, on the seventh term of the left-hand side of (\ref{eq4-23}), the following estimate holds:
\begin{align}
& \varphi'(s(t))s_t(t)|\tilde{u}(t, 1)|^2 \nonumber \\
=& \varphi'(s(t)) \left(\frac{1}{a_0}|s_t(t)|^2 + \varphi(s(t))s_t(t)\right) \tilde{u}(t, 1) \nonumber \\
=& \frac{\varphi'(s(t))}{a_0}|s_t(t)|^2 \tilde{u}(t, 1) + \varphi'(s(t))\varphi(s(t))s_t(t)\tilde{u}(t, 1) \nonumber \\
=&  \frac{\varphi'(s(t))}{a_0}|s_t(t)|^2 \tilde{u}(t, 1) + \varphi'(s(t))\varphi(s(t))\left(\frac{1}{a_0}|s_t(t)|^2 + \varphi(s(t))s_t(t)\right) \nonumber \\
=& \frac{\varphi'(s(t))}{a_0}|s_t(t)|^2 \tilde{u}(t, 1) + \frac{\varphi'(s(t))\varphi(s(t))}{a_0}|s_t(t)|^2 + \varphi'(s(t))(\varphi(s(t)))^2s_t(t) \nonumber \\
=& \frac{\varphi'(s(t))}{a_0}|s_t(t)|^2 \tilde{u}(t, 1) + \frac{\varphi'(s(t))\varphi(s(t))}{a_0}|s_t(t)|^2 + \frac{1}{3}\frac{d}{dt}(\varphi(s(t))^3, \nonumber 
\end{align}
namely, 
\begin{align}
\label{eqb-17}
\frac{a_0}{2}\int_0^{t_1}|\tilde{u}(t, 1)|^2\varphi'(s(t))s_t(t) dt \geq \frac{a_0}{2}\left(\frac{1}{3} (\varphi(s(t_1)))^3-\frac{1}{3}(\varphi(s_0))^3 \right) \mbox{ for }t_1\in [0, T].
\end{align}
Therefore, (\ref{eq4-23}) combining with (\ref{eqb-17}) gives  
\begin{align}
\label{eq4-24}
& \int_0^{t_1}\int_a^{s(t)} |u_t(t)|^2 + \frac{k}{2}\int_a^{s(t_1)}|u_z(t_1)|^2dz  +\frac{a_0}{6}(\varphi(s(t_1)))^3 \nonumber \\
\leq & \frac{k}{2}\int_a^{s_0}|u_z(0)|^2dz +  |\Phi(s(t_1), \tilde{u}(t_1, 1))| + |\Phi(s(0), \tilde{u}_0(1))|  \nonumber \\
& + |\hat{\beta}(b(t_1), \tilde{u}(t_1, 0))|  +  \frac{a_0}{6}(\varphi(s_0))^3  + |\beta'|_{L^{\infty}(\mathbb{R})}\frac{b^*}{H} \int_0^{t_1}|b_t(t)| dt\nonumber \\
& -\int_0^{t_1} \frac{k}{2}|u_z(t, s(t))|^2s_t(t) dt \mbox{ for }t_1\in [0, T]. 
\end{align}
What concerns  the right-hand side of (\ref{eq4-24}), by (1.3), (1.4), (A3), (A4), (A5), $0\leq \varphi(a) \leq \tilde{u}(t)\leq b^*H^{-1}$ on $[0, 1]$ for $t\in [0, T]$ and the definition of $\Phi$ and $\hat{\beta}$, we proceed as follows:
\begin{align}
\label{eq4-25}
|\Phi(s(t_1), \tilde{u}(t_1, 1))| & =\biggl | a_0\left(\frac{\tilde{u}^3(t_1, 1)}{3}-\frac{\varphi(s(t_1))}{2}\tilde{u}^2(t_1, 1)\right) \biggr| \nonumber \\
& \leq \frac{a_0}{3}|\tilde{u}^3(t_1, 1)|\leq \frac{a_0}{3}\left(\frac{b^*}{H} \right)^3,
\end{align}
\begin{align}
\label{eq4-26}
|\Phi(s(0), \tilde{u}_0(1))| \leq \frac{a_0}{3}|\tilde{u}^3_0(1)|\leq \frac{a_0}{3}\left(\frac{b^*}{H} \right)^3, 
\end{align}
\begin{align}
\label{eq4-27}
|\hat{\beta}(b(t_1), \tilde{u}(t_1, 0))| \leq \int_0^{\tilde{u}(t_1, 0)}\beta(b(t)-H\tau)d\tau \leq k_0 \frac{b^*}{H}, 
\end{align}
and 
\begin{align}
\label{eq4-29}
&-\int_0^{t_1} \frac{k}{2}|u_z(t, s(t))|^2s_t(t) dt \leq \frac{k}{2}\int_0^{t_1}\biggl|\frac{1}{k}(u(t, s(t))s_t(t) \biggr|^2 |s_t(t)| dt \nonumber \\
\leq & \frac{a^3_0}{2k}\int_0^{t_1}|u(t, s(t))|^5 dt \leq \frac{a^3_0}{2k}\left(\frac{b^*}{H} \right)^5T. 
\end{align}
Finally, by applying all estimates (\ref{eq4-25})-(\ref{eq4-29}) to (\ref{eq4-24}) and $b\in W^{1,2}(0, T)$ as in (A2) we see that there exists $C$ which depends on $s_0$, $b^*$, $a_0$, $H$, $c_{\varphi}$, $k_0$ and $T$ such that (\ref{eq4-2}) holds. Thus, Lemma \ref{lem3-1} is proved. 
\end{proof}

\section{Proof of Theorems \ref{th3} and \ref{b-4}}

\subsection{Global existence} 

In this section, we prove Theorem \ref{th3} -- our global existence result.  
First, we note that (P)$_{\tilde{u}_0, s_0, b}$ has a unique solution locally in time as shown in \cite{KA}. 

\begin{lem}(Local solvability; cf. \cite{KA})
\label{lem3}
Assume (A1)-(A5). Then, for $T>0$ and $L>s_0$ there exists $T'<T$ such that (P)$_{\tilde{u}_0, s_0, b}$ has a unique solution $(s, \tilde{u})$ on $[0, T']$ satisfying $s<L$ on $(0, T']$ and $\varphi(a)\leq \tilde{u} \leq b^*H^{-1}$ on $Q(T')=(0, T')\times (0, 1)$.
\end{lem}

Let $T>0$ and  put $\tilde{L}(T)=a_0b^*H^{-1}T +s_0$. 
Clearly, $a_0b^*H^{-1}T +s_0>s_0$ so that 
by Lemma \ref{lem3} there exists $T_1<T$ such that (P)$_{\tilde{u}_0, s_0, b}$ has a unique solution $(s, \tilde{u})$ on $[0, T_1]$ satisfying $s<\tilde{L}(T)$ on $(0, T_1]$ and $\varphi(a)\leq \tilde{u} \leq b^*H^{-1}$ on $(0, T_1)\times (0, 1)$. Let us take  
\begin{align*}
\tilde{T}:=\mbox{sup}\{T_1>0 |  (\mbox{P})_{\tilde{u}_0, s_0, b} \mbox{ has a solution }(s, \tilde{u}) \mbox{ on } [0, T_1]\}.
\end{align*}
By the local existence result, we deduce that $\tilde{T}>0$. Now. we assume $\tilde{T}<T$. 
Then, by (\ref{p2-4}) we obtain that 
\begin{align}
\label{eq5-1}
s(t) &=\int_0^{t} a_0(\tilde{u}(t, 1)-\varphi(s(t))) dt +s_0 \nonumber \\
&\leq a_0\frac{b^*}{H} \tilde{T}+s_0 =\tilde{L}(\tilde{T}) <\tilde{L}(T)\mbox{ for }t\in [0, \tilde{T}). 
\end{align}
Next, similarly to the derivation of (\ref{eq4-15}), we have that 
\begin{align}
& \int_0^t\int_0^1|\tilde{u}_t(\tau, y)|^2dyd\tau \nonumber \\
=&\int_0^t\int_a^{s(\tau)}\left(\frac{1}{s(\tau)-a}\right)|u_t(\tau, z)+u_z(\tau, z)(s(\tau)-a)|^2dzd\tau \nonumber \\
\leq &\int_0^t\int_a^{s(\tau)}\left(\frac{1}{s(\tau)-a}\right)|u_t(t, z)|^2dzd\tau + \int_0^t\int_a^{s(\tau)}2|u_t(\tau, z)||u_z(\tau, z)|dzd\tau \nonumber \\
& + \int_0^t\int_a^{s(\tau)}(s(\tau)-a)|u_z(\tau, z)|^2dzd\tau  \label{eq5-2} 
\end{align}
as well as 
\begin{align}
\label{eq5-3}
& \int_0^1|\tilde{u}_y(t)|^2dy =\int_a^{s(t)}\frac{1}{s(t)-a}|u_z(t) (s(t)-a)|^2dzdt.
\end{align}
Therefore, from (\ref{eq5-2}), (\ref{eq5-3}) jointly with Lemma \ref{lem2} and Lemma \ref{lem3-1},  it follows that 
\begin{align}
\label{eq5-4}
\int_0^{t}|\tilde{u}_t(\tau)|^2_{L^2(0, 1)} d\tau \leq \biggl( \left(\frac{1}{s_*-a} +1 \right) + (\tilde{L}(T)-a+1)T\biggr) C(T)
\mbox{ for }t\leq \tilde{T}, 
\end{align}
and 
\begin{align}
\label{eq5-5}
|\tilde{u}_y(t)|^2_{L^2(0, 1)}\leq (\tilde{L}(T)-a)C(T) \mbox{ for }t\leq \tilde{T}, 
\end{align}
where $s_*$ is the same as in Lemma \ref{lem2}. 
By (\ref{eq5-4}) and (\ref{eq5-5}) we observe that for some $\tilde{u}_{\tilde{T}}\in H^1(0, 1)$, 
$\tilde{u}(t)\to \tilde{u}_{\tilde{T}}$ strongly in $L^2(0, 1)$ and weakly in $H^1(0, 1)$
as $t\to \tilde{T}$. Then, it is clear that $\varphi(a)\leq \tilde{u}_{\tilde{T}} \leq b^*H^{-1}$. Also, by (\ref{p2-4}) and $\varphi(a)\leq \tilde{u} \leq b^*H^{-1}$ on $[0, 1]$ for $t\in [0, \tilde{T})$ we have $|s_t(t)|\leq a_0b^*H^{-1}$ for $t\in [0, \tilde{T})$. This implies that $\{s(t)\}_{t\in [0, \tilde{T})}$ is a Cauchy sequence in $\mathbb{R}$ so that $s(t) \to s_{\tilde{T}}$ in $\mathbb{R}$ as $t\to \tilde{T}$, and from Lemma \ref{lem2}, $s_{\tilde{T}}$ satisfies that $a<s_{\tilde{T}}\leq a_0b^*H^{-1}\tilde{T}+s_0$. Now, we put $u_{\tilde{T}}(z)=\tilde{u}_{\tilde{T}}(\frac{z-a}{s_{\tilde{T}}-a})$ for $z\in [a, s_{\tilde{T}}]$. Then,  we see that $u_{\tilde{T}}\in H^1(a, s_{\tilde{T}})$ and $\varphi(a)\leq u_{\tilde{T}}\leq b^*H^{-1}$ and we can consider $(s_{\tilde{T}}, u_{\tilde{T}})$ as a initial data.
Therefore, by replacing $\tilde{L}(T)$ by $\tilde{L}(T)+ s_{\tilde{T}}$ and repeating the argument of the local existence we  obtain a unique solution of (P)$_{\tilde{u}_{\tilde{T}}, s_{\tilde{T}}, b}$ which implies that a solution of (P)$_{\tilde{u}_0, s_0, b}$ can be extended beyond $\tilde{T}$. This is a contradiction for the definition of $\tilde{T}$, and therefore, we see that $\tilde{T}=T$. Thus, we can show the existence of a unique solution of (P)$_{\tilde{u}_{\tilde{T}}, s_{\tilde{T}}, b}$ on $[0, T]$, and thus Theorem \ref{th3} is proven. 


\subsection{Continuous dependence estimates} 

In this section, we prove Theorem \ref{b-4} on the continuous dependence estimate of the solution $(s, u)$ of (P)$_{u_0,s_0, b}$ with respect to the given data $u_0$, $s_0$ and $b$. The working strategy is here standard, see e.g. \cite{AM} for a related result. 
As departure point, we note from Lemma \ref{lem2} that the free boundary $s$ has a lower bound $s_*$ on $[0, T]$ for any $T>0$. 
Also, we see from (\ref{eq5-1}) that $s(t) \leq \tilde{L}(T)$ for $t\in [0, T]$, where $\tilde{L}(T)$ is the same definition of (\ref{eq5-1}).  
For simplicity, we put $\tilde{u}=\tilde{u}_1-\tilde{u}_2$ and $s=s_1-s_2$ and allow the comparison to hold on the same time interval, say $[0,T]$.  
Since $(s_i, \tilde{u}_i)$ for $i=1, 2$ satisfies 
\begin{align}
\tilde{u}_{it}(t, y)-\frac{k}{(s_i(t)-a)^2}\tilde{u}_{iyy}(t, y)=\frac{ys_{it}(t)}{s_i(t)-a}\tilde{u}_{iy}(t, y) \mbox{ for }(t, y)\in Q(T), \nonumber 
\end{align}
we can obtain 
\begin{align}
\label{eqb-28}
& \frac{1}{2} \frac{d}{dt} |\tilde{u}(t)|^2_{L^2(0, 1)} - \int_0^1 \left( \frac{k}{(s_1(t)-a)^2} \tilde{u}_{1yy}(t) - \frac{k}{(s_2(t)-a)^2} \tilde{u}_{2yy}(t) \right) \tilde{u}(t) dy \nonumber \\
& = \int_0^1 \left( \frac{ys_{1t}(t)}{s_1-a} \tilde{u}_{1y}(t) - \frac{ys_{2t}(t)}{s_2-a} \tilde{u}_{2y}(t) \right) \tilde{u}(t) dy. 
\end{align}
On the second term of the left-hand side of (\ref{eqb-28}), by integration by parts, we observe that 
\begin{align}
\label{eqb-29}
& - \int_0^1 \left( \frac{k}{(s_1(t)-a)^2} \tilde{u}_{1yy} - \frac{k}{(s_2(t)-a)^2} \tilde{u}_{2yy} \right) \tilde{u} dy \nonumber \\
= &  -\left(\frac{k}{(s_1(t)-a)^2} \tilde{u}_{1y}(t, 1) - \frac{k}{(s_2(t)-a)^2} \tilde{u}_{2y}(t, 1) \right)\tilde{u}(t, 1) \nonumber \\
& +  \left( \frac{k}{(s_1(t)-a)^2} \tilde{u}_{1y}(t, 0) - \frac{k}{(s_2(t)-a)^2} \tilde{u}_{2y}(t, 0) \right)\tilde{u}(t, 0) \nonumber \\
& + \int_0^1 \left( \frac{k}{(s_1(t)-a)^2} \tilde{u}_{1y}(t) - \frac{k}{(s_2(t)-a)^2} \tilde{u}_{2y}(t) \right) \tilde{u}_y(t) dy.
\end{align}
Now, we denote the right-hand side of (\ref{eqb-29}) by $I_1$, $I_2$ and $I_3$. The term $I_1$ can be transformed as 
\begin{align}
I_1(t) &= \left( \frac{\tilde{u}_1(t, 1)s_{1t}(t)}{s_1(t)-a} - \frac{\tilde{u}_2(t, 1)s_{2t}(t)}{s_2(t)-a} \right) \tilde{u}(t, 1) \nonumber \\
& = \frac{s_{1t}(t)}{s_1(t)-a} |\tilde{u}(t, 1)|^2 + \left( \frac{\tilde{u}_2(t, 1)s_{1t}(t)}{s_1(t)-a} - \frac{\tilde{u}_2(t, 1)s_{2t}(t)}{s_2(t)-a} \right) \tilde{u}(t, 1) \nonumber \\
& =\frac{s_{1t}(t)}{s_1(t)-a} |\tilde{u}(t, 1)|^2 + \frac{\tilde{u}_2(t, 1)}{s_1(t)-a}s(t)\tilde{u}(t, 1) \nonumber \\
&  \quad + \tilde{u}_2(t, 1)s_{2t}(t)\left(\frac{1}{s_1(t)-a} -\frac{1}{s_2(t)-a} \right) \tilde{u}(t, 1). \nonumber 
\end{align}
Accordingly, by the properties for solutions on $[0, T]$: $|s_{it}| \leq a_0 b^*H^{-1}$, $s_* \leq s_i\leq \tilde{L}(T)$ and $\tilde{u}_i(\cdot, 1) \leq b^*H^{-1}$ it follows that 
\begin{align}
\label{eqb-31}
I_1(t) & \geq -M_1(|\tilde{u}(t, 1)|^2 + |s(t)|^2) \mbox{ for a.e. }t\in [0, T], 
\end{align}
where $M_1$ is a positive constant. Next, the term $I_2$ can be dealt as 
\begin{align}
I_2(t) &= -\left( \frac{\beta(b_1(t)-H\tilde{u}_1(t, 0))}{s_1(t)-a} -\frac{\beta(b_2(t)-H\tilde{u}_2(t, 0))}{s_2(t)-a} \right) \tilde{u}(t, 0) \nonumber \\
&= -\frac{1}{s_1(t)-a}\biggl(\beta(b_1(t)-H\tilde{u}_1(t, 0))-\beta(b_2(t)-H\tilde{u}_2(t, 0))\biggr)\tilde{u}(t, 0) \nonumber \\
& \quad -\beta(b_2(t)-H\tilde{u}_2(t, 0))\left(\frac{1}{s_1(t)-a}-\frac{1}{s_2(t)-a} \right)\tilde{u}(t, 0). \nonumber 
\end{align}
Note that by the Lipshitz continuity of $\beta$, 
\begin{align}
\label{eqb-33}
|\beta(b_1(t)-H\tilde{u}_1(t, 0))-\beta(b_2(t)-H\tilde{u}_2(t, 0))| \leq |\beta'|_{L^{\infty}(\mathbb{R})}(|b(t)|+ H|\tilde{u}(t, 0)|).
\end{align}
Since $\beta \leq k_0$ and $s_* \leq s_i \leq \tilde{L}(T)$ on $[0, T]$ we obtain the following estimate on the term $I_2$:
\begin{align}
\label{eqb-34}
I_2(t) \geq -M_2(|b(t)|^2 + |\tilde{u}(t, 0)|^2 + |s(t)|^2) \mbox{ for a.e. }t\in [0, T], 
\end{align}
where $M_2$ is a positive constant. 
On the term $I_3$, we can estimate as follows:
\begin{align}
\label{eqb-35}
I_3(t) &=  \frac{k}{(s_1(t)-a)^2} \int_0^1 |\tilde{u}_{y}(t)|^2 dy + \left( \frac{k}{(s_1(t)-a)^2}-\frac{k}{(s_2(t)-a)^2}\right) \int_0^1 \tilde{u}_{2y}(t)\tilde{u}_y(t) dy \nonumber \\
\geq & \frac{k}{(s_1(t)-a)^2} \int_0^1 |\tilde{u}_{y}(t)|^2 dy - \frac{k}{4(s_1(t)-a)^2} \int_0^1 |\tilde{u}_y(t)|^2 dy -\frac{k(2\tilde{L}(T))^2}{(s_*-a)^6}|s(t)|^2|\tilde{u}_{2y}(t)|^2_{L^2(0, 1)} \nonumber \\
\geq & \frac{3k}{4(s_1(t)-a)^2} \int_0^1 |\tilde{u}_{y}(t)|^2 dy -\frac{k(2\tilde{L}(T))^2C(T)}{(s_*-a)^6}|s(t)|^2 \mbox{ for a.e. }t\in [0, T],
\end{align}
where $C(T)$ is the same constant as in Lemma \ref{lem3-1}.  Also, we deal the right-hand side of (\ref{eqb-28}) which is denoted by $I_4$ as follows:
\begin{align}
\label{eqb-36}
|I_4(t)| &\leq \frac{|s_{1t}(t)|}{s_1(t)-a} \int_0^1 |\tilde{u}_y(t)||\tilde{u}(t)| dy \nonumber \\
& + \frac{|s_{1t}(t)||s(t)|}{(s_1(t)-a)(s_2(t)-a)} \int_0^1 |\tilde{u}_{2y}(t)||\tilde{u}(t)| dy \nonumber \\
& + \frac{|s_{t}(t)|}{s_2(t)-a} \int_0^1 |\tilde{u}_{2y}(t)||\tilde{u}(t)| dy \mbox{ for a.e. }t\in [0, T].
\end{align}
Since $|s_t(t)| \leq a_0(|\tilde{u}(t, 1)| + C_{\varphi}|s(t)|)$, by using $s_* \leq s_i\leq \tilde{L}(T)$ on $[0, T]$ and (\ref{eq5-3}), we have 
\begin{align}
\label{eqb-37}
|I_4(t)| \leq \frac{k}{8(s_1(t)-a)^2}|\tilde{u}_y(t)|^2_{L^2(0, 1)} + M_3(|s(t)|^2 + |\tilde{u}(t)|^2+|\tilde{u}(t, 1)|^2) \mbox{ for a.e. }t\in[0, T], 
\end{align}
where $M_3$ is a positive constant. From all estimates (\ref{eqb-31}), (\ref{eqb-34}), (\ref{eqb-35}) and (\ref{eqb-37}), we obtain from (\ref{eqb-28}) that 
\begin{align}
\label{eqb-38}
& \frac{1}{2} \frac{d}{dt} |\tilde{u}(t)|^2_{L^2(0, 1)} + \frac{5k}{8(s_1(t)-a)^2}|\tilde{u}_y(t)|^2_{L^2(0, 1)} \nonumber \\
\leq & M_4( |s(t)|^2  + |b(t)|^2 + |\tilde{u}(t, 0)|^2  + |\tilde{u}(t)|^2+|\tilde{u}(t, 1)|^2) \mbox{ for a.e. } t\in [0, T],
\end{align}
where $M_4$ is a positive constant depending on $M_1$, $M_2$ and $M_3$. 
Now, we give the estimate of $s$ by 
\begin{align}
\label{eqb-39}
& \frac{1}{2} \frac{d}{dt} |s(t)|^2 \leq a_0|\tilde{u}(t, 1)-(\varphi(s_1(t))-\varphi(s_2(t)))||s(t)| \nonumber \\
\leq & \frac{a_0}{2} (|\tilde{u}(t, 1)|^2 + |s(t)|^2 ) + a_0C_{\varphi}|s(t)|^2. 
\end{align}
By applying (\ref{eqb-39}) to (\ref{eqb-38}) we have that 
\begin{align}
\label{eqb-40}
& \frac{1}{2} \frac{d}{dt} |\tilde{u}(t)|^2_{L^2(0, 1)} + \frac{1}{2} \frac{d}{dt} |s(t)|^2 + \frac{5k}{8(s_1(t)-a)^2}|\tilde{u}_y(t)|^2_{L^2(0, 1)} \nonumber \\
\leq & M_5( |s(t)|^2 + |b(t)|^2 + |\tilde{u}(t, 0)|^2  + |\tilde{u}(t)|^2+|\tilde{u}(t, 1)|^2) \mbox{ for a.e. } t\in [0, T],
\end{align}
where $M_5$ is a positive constant depending on $M_4$, $a_0$ and $C_\varphi$.
Here, the Sobolev's embedding theorem in one dimension guarantees that 
\begin{align}
\label{eqb-41}
& |\tilde{u}(t, z)|^2  \leq C_e |\tilde{u}(t)|_{H^1(0, 1)}|\tilde{u}(t)|_{L^2(0, 1)} \nonumber \\
\leq & C_e (|\tilde{u}_y(t)|_{L^2(0, 1)}|\tilde{u}(t)|_{L^2(0, 1)} + |\tilde{u}(t)|^2_{L^2(0, 1)}) \nonumber \\
\leq & \frac{k}{16M_5(s_1(t)-a)^2}|\tilde{u}_y(t)|^2_{L^2(0, 1)} + \left(\frac{4M_5(\tilde{L}(T)-a)^2}{k}C^2_e + C_e\right) |\tilde{u}(t)|^2_{L^2(0, 1)} \nonumber \\
& \mbox{ at } y=0, 1 \mbox{ for a.e. }t\in [0, T],
\end{align}
where $C_e$ is the constant by Sobolev's embedding theorem. Consequently, by (\ref{eqb-40}) and (\ref{eqb-41}) we obtain that 
\begin{align}
\label{eqb-42}
& \frac{1}{2} \frac{d}{dt} |\tilde{u}(t)|^2_{L^2(0, 1)} + \frac{1}{2} \frac{d}{dt} |s(t)|^2 + \frac{k}{2(s_1(t)-a)^2}|\tilde{u}_y(t)|^2_{L^2(0, 1)} \nonumber \\
\leq & M_6( |s(t)|^2 + |b(t)|^2 + |\tilde{u}(t)|^2_{L^2(0, 1)}) \mbox{ for a.e. } t\in [0, T],
\end{align}
where $M_6$ is a positive constant. Therefore, Gronwall's inequality implies that there exists $M_T>0$ such that (\ref{eqb-22}) holds. Thus, Theorem \ref{b-4} is proved.


\section{Large time behavior of the free boundary} \label{large}

In this section, we discuss the large time behavior of a solution to (P)$_{u_0, s_0, b}$ as $t\to \infty$. We introduce firstly the auxiliary assumption 
(A2)' as follows: 
\newline
(A2)': $b\in W^{1, 2}_{loc}([0, \infty))$, $b_t\in L^1(0, \infty)$, lim$_{t\to \infty} b(t)=b_{\infty}$, $b-b_{\infty}\in L^1(0, \infty)$ and 
$b_*\leq b\leq b^*$ on $(0, \infty)$, where $b_*$ and $b^*$ are positive constants as in (A2). 
\newline
Clearly, we see that $b_*\leq b_{\infty}\leq b^*$. Next, we consider the following stationary problem (P)$_{\infty}$: find a pair $(u_{\infty}, s_{\infty}) \in L^2(a, s_{\infty}) \times \mathbb{R}$ satisfying 
$$
\begin{cases}
-ku_{\infty zz}=0 \mbox{ on } (a, s_{\infty}), \\
-ku_{\infty z}(a)=\beta(b_{\infty}-Hu_{\infty}(a)), \ -ku_{\infty z}(s_{\infty})=0,\\ 
u_{\infty}(s_{\infty})=\varphi(s_{\infty}).
\end{cases}
$$
By using the change of variable $\tilde{u}_{\infty}(y)=u_{\infty}((1-y)a+ys_{\infty})$ for $y\in (0, 1)$, it is easy to see that (P)$_{\infty}$ is equivalent to the following problem $(\tilde{\mbox{P}})_{\infty}$:
$$
\begin{cases}
-\displaystyle{\frac{k}{(s_{\infty}-a)^2}}\tilde{u}_{\infty yy}=0 \mbox{ on } (0, 1), \\[2mm]
-\displaystyle{\frac{k}{s_{\infty}-a}}\tilde{u}_{\infty y}(0)=\beta(b_{\infty}-H\tilde{u}_{\infty}(0)), \ -\displaystyle{\frac{k}{s_{\infty}-a}}\tilde{u}_{\infty y}(1)=0,\\[2mm] 
\tilde{u}_{\infty}(1)=\varphi(s_{\infty}).
\end{cases}
$$
The next lemma is concerned with the existence and the structure of a solution to the problem $(\tilde{\mbox{P}})_{\infty}$. 
\begin{lem}
\label{b-1} 
Let $W$ be the set of $s_{\infty}$ given by 
\begin{align}
W:=\{ s_{\infty} \in \mathbb{R} \ | \ s_{\infty} >a \mbox{ with } \varphi(s_{\infty})\geq b_{\infty}H^{-1} \}. \nonumber 
\end{align} 
Then,  $(s_{\infty}, \tilde{u}_{\infty})$ is a solution of $(\tilde{\mbox{P}})_{\infty}$ satisfying $s_{\infty}>a$ and $\tilde{u}_{\infty}\in H^2(a, s_{\infty})$ if and only if $s_{\infty}\in W$ and $\tilde{u}_{\infty}(y)=\varphi(s_{\infty})$ for $y\in (0, 1)$. 
\end{lem}
The proof of Lemma \ref{b-1} is trivial, hence we omit to show it.  
Here, we note that if $c_0<b_*H^{-1}$, then $W=\phi$. Indeed, if $c_0<b_*H^{-1}$, then $\varphi(r) \leq c_0 < b_*H^{-1} \leq b_{\infty}H^{-1}$ for $r\in \mathbb{R}$ so that 
there does not exist $r\in \mathbb{R}$ such that $\varphi(r) \geq b_{\infty}H^{-1}$. 
On the other hand, $b^*H^{-1}<c_0$ is a sufficient condition for $W \neq \phi$. 
Indeed, if $b^*H^{-1}<c_0$, then $b_{\infty}H^{-1 }\leq b^*H^{-1}<c_0$, and therefore, the mean value theorem guarantees that the existence of $s_{\infty}$ such that $s_{\infty}>a$ and $\varphi(s_{\infty})\geq b_{\infty}H^{-1}$. 
\newline
Now, we are able to state the result on the large time behavior of a solution as $t\to \infty$. Here, (A4)' is the following condition:
\newline
(A4)': In (A4) replace (\ref{c0-1}) by 
\begin{align}
\label{a4-1}
 0<c_0<\mbox{min}\{2\varphi(a), b_* H^{-1}\}.
\end{align}
Also, we put $r_*:=a+H(b^*)^{-1}|u_0|_{L^1(a, s_0)}$. Then, by $u_0 \leq b^*H^{-1}$ on $[a, s_0]$ in (A5), we note that $r_* \leq a +(s_0-a)<r_{\varphi}$. The following dichotomy becomes now available: 
\begin{thm}
\label{b-2}
Assume (A1), (A2)', (A3), (A4)' and (A5) and let (P)$_{u_0, s_0, b}$ be a solution $(s, u)$ on $[0, \infty)$. 
Then, either of the following holds:
\begin{align}
(i) & \lim_{t \to \infty} \int_0^t \beta(b(\tau)-Hu(\tau, a))d\tau=\infty \mbox{ and } \lim_{t\to \infty}s(t)=\infty. \nonumber \\
(ii) &\lim_{t \to \infty} \int_0^t \beta(b(\tau)-Hu(\tau, a))d\tau \leq  \frac{b^*}{H}(r_{\varphi}-r_*) \mbox{ and } \limsup_{t \to \infty}s(t)\leq a+ \frac{1}{\varphi(a)} \frac{b^*}{H}(s_0-a+r_{\varphi}-r_*). \nonumber 
\end{align}
\end{thm}
This result also gives a partial answer to question (Q2). We handle its proof in the next subsections. 
\subsection{Uniform estimates with respect to time}
To prove Theorem \ref{b-2}, we provide some uniform estimates for the solution with respect to time $t$. 
Throughout of the section we assume (A1), (A2)' (A3), (A4)', (A5). By Theorem \ref{th2}, we note that (P)$_{u_0, s_0, b}$ has a solution $(s, u) $ on $[0, T]$ for $T>0$ satisfying 
$\varphi(a)\leq u\leq b^*H^{-1}$ on $[a, s(t)]$ for $t\in [0, T]$. 
\begin{lem}
\label{b-3}
Let $(s, u)$ be a solution of  (P)$_{u_0, s_0, b}$ on $[0, \infty)$ and put $F(t)=\int_0^t \beta(h(\tau)-Hu(\tau, a))d\tau$ for any $t>0$. 
If there exists $t_0>0$ such that $F(t_0)>b^*H^{-1}(r_{\varphi}-r_*)$, 
then there exist $C_1>0$ and $C_2>0$ which is independent of time $t$ such that  
\begin{align}
(i) & \int_0^t |s_t(\tau)|^2 d\tau + \int_0^t |u_z(\tau)|^2_{L^2(a, s(\tau))} d\tau \leq C_1(1+F(t)) \mbox{ for }t>0, \label{eqb0-1}\\
(ii) & \int_{0}^t|u_t(\tau)|^2_{L^2(a, s(\tau))}d\tau + |u_z(t)|^2_{L^2(a, s(t))} \leq C_2(1+F(t)) \mbox{ for }t>0. \label{eqb0-2} 
\end{align}
\end{lem}
\begin{proof}
First, we show that $s(t)>r_{\varphi}$ for any $t\geq t_0$. By the derivation of the lower bound of $s$ in Lemma \ref{lem2}, we have that 
\begin{align}
\label{eqb-00}
s(t) & \geq a + \frac{H}{b^*} \left (\int_a^{s_0}u_0 dz + \int_0^t \beta(b(\tau)-Hu(\tau, a)) d\tau \right) \mbox{ for }t>0. 
\end{align}
Then, from the monotonicity of $F$ with respect to time $t$ and $F(t_0)>b^*H^{-1}(r_{\varphi}-r_*)$, we see that 
\begin{align}
\label{eqb-0}
s(t) & \geq a + \frac{H}{b^*} \left (\int_a^{s_0}u_0 dz  + \int_0^{t_0} \beta(b(\tau)-Hu(\tau, a)) d\tau \right) \nonumber \\
& >a + \frac{H}{b^*} \left (\int_a^{s_0}u_0 dz  + \frac{b^*}{H}(r_{\varphi}-r_*) \right) \nonumber \\
&> a + \frac{H}{b^*} \int_a^{s_0}u_0 dz  + (r_{\varphi}-r_*) \mbox{ for }t\geq t_0. 
\end{align}
Since $r_*=a+H (b^*)^{-1}|u_0|_{L^1(a, s_0)}$, (\ref{eqb-0}) implies that $s(t)>r_{\varphi}$ for $t\geq t_0$. 
Using this result, we prove (i). By integrating $[a, s(t)]$ over (\ref{p1-1}) with the boundary conditions (\ref{p1-2}) and (\ref{p1-3}) we have that 
\begin{align}
\label{eqb-1}
& \frac{1}{2} \frac{d}{dt} \int_a^{s(t)}|u(t)|^2dz-\frac{s_t(t)}{2}|u(t, s(t))|^2  + k\int_a^{s(t)}|u_z(t)|^2 dz \nonumber \\
& + u(t, s(t))s_t(t) u(t, s(t)) - \beta(b(t)-Hu(t, a))u(t, a) \nonumber \\
& =0 \mbox{ for a.e. }t\in [0, T].
\end{align}
On the second and forth term in the left-hand side of (\ref{eqb-1}) by using $s_t=a_0(u(t, s(t))-\varphi(s(t)))$ we can deal as 
\begin{align}
\label{eqb-2}
& -\frac{s_t(t)}{2}|u(t, s(t))|^2 + u(t, s(t))s_t(t) u(t, s(t)) = \frac{s_t(t)}{2}|u(t, s(t))|^2 \nonumber \\
= & \frac{1}{2} \left(\frac{|s_t(t)|^2}{a_0} + \varphi(s(t))s_t(t) \right)u(t, s(t)) \nonumber \\
= & \frac{1}{2a_0}|s_t(t)|^2 u(t, s(t)) + \frac{1}{2}\varphi(s(t))s_t(t)u(t, s(t)) \nonumber \\
= & \frac{1}{2a_0}|s_t(t)|^2 u(t, s(t)) + \frac{1}{2}\varphi(s(t))\left(\frac{|s_t(t)|^2}{a_0} + \varphi(s(t))s_t(t) \right) \nonumber \\
= & \frac{1}{2a_0}|s_t(t)|^2 u(t, s(t)) + \frac{1}{2a_0}|s_t(t)|^2 \varphi(s(t)) + \frac{1}{2}\varphi^2(s(t))s_t(t). 
\end{align}
Here, we note that $\varphi(s(t))=c_0$ for $t\geq t_0$ because $s(t)>r_{\varphi}$ for $t\geq t_0$. Therefore, we have that 
\begin{align}
\label{eqb-3}
& -\frac{s_t(t)}{2}|u(t, s(t))|^2 + u(t, s(t))s_t(t) u(t, s(t)) \nonumber \\
= & \frac{1}{2a_0}|s_t(t)|^2 u(t, s(t)) + \frac{1}{2a_0}|s_t(t)|^2 c_0 + \frac{c^2_0}{2}s_t(t) \mbox{ for }t\geq t_0.
\end{align}
By (\ref{eqb-1}) combining with (\ref{eqb-3}) we have that 
\begin{align}
\label{eqb-4}
& \frac{1}{2} \frac{d}{dt} \int_a^{s(t)}|u(t)|^2dz +  k\int_a^{s(t)}|u_z(t)|^2 dz \nonumber \\
& + \frac{1}{2a_0}|s_t(t)|^2 u(t, s(t)) + \frac{1}{2a_0}|s_t(t)|^2 c_0 + \frac{c^2_0}{2}s_t(t) \nonumber \\
=& \beta(b(t)-Hu(t, a))u(t, a) \mbox{ for a.e. }t\geq t_0. 
\end{align}
Hence, by integrating (\ref{eqb-4}) over $[t_0, t]$ for $t>t_0$ and using $u\leq b^*H^{-1}$ on $[a, s(t)]$ for $t>0$, we obtain that 
\begin{align}
\label{eqb-5}
& \frac{1}{2} \int_a^{s(t)}|u(t)|^2dz +  k\int_{t_0}^t \int_a^{s(\tau)}|u_z(\tau)|^2 dzd\tau +\frac{\varphi(a)+c_0}{2a_0}\int_{t_0}^t|s_t(\tau)|^2d\tau  + \frac{c^2_0}{2}s(t)  \nonumber \\
& \leq \frac{c^2_0}{2} s(t_0) + \frac{1}{2} \int_a^{s(t_0)}|u(t_0)|^2dz + \frac{b^*}{H} \int_{t_0}^t \beta(b(\tau)-Hu(\tau, a))d\tau.
\end{align}
Also, the integration of (\ref{eqb-1}) combining with (\ref{eqb-2}) on $[0, t]$ for $t\leq t_0$ gives 
\begin{align}
\label{eqb-6}
& \frac{1}{2} \int_a^{s(t)}|u(t)|^2dz +  k\int_{0}^t \int_a^{s(\tau)}|u_z(\tau)|^2 dzd\tau +\frac{\varphi(a)}{a_0}\int_0^t|s_t(\tau)|^2d\tau \nonumber \\
& \leq \frac{1}{2} \int_a^{s_0}|u_0|^2dz +\frac{c^2_0}{2}\int_0^{t}|s_t(\tau)|d\tau + \frac{b^*}{H} \int_{0}^t \beta(b(\tau)-Hu(\tau, a))d\tau.
\end{align}
The inequalities (\ref{eqb-5}) and (\ref{eqb-6}) guarantee that there exists $C_1>0$ such that (\ref{eqb0-1}) holds. 
\newline
Next, we show (\ref{eqb0-2}).  This result is obtained as the proof of Lemma \ref{lem3-1}. Indeed, by repeating the same derivation of (\ref{eq4-23}) we can have that 
\begin{align}
\label{eqb-7}
& \int_0^{t_1}\int_a^{s(t)} |u_t(t)|^2 + \frac{k}{2}\int_a^{s(t_1)}|u_z(t_1)|^2dz  +\frac{a_0}{6}(\varphi(s(t_1)))^3 \nonumber \\
\leq & \frac{k}{2}\int_a^{s_0}|u_z(0)|^2dz +  |\Phi(s(t_1), \tilde{u}(t_1, 1))| + |\Phi(s(0), \tilde{u}_0(1))|  \nonumber \\
& + |\hat{\beta}(b(t_1), \tilde{u}(t_1, 0))|  +  \frac{a_0}{6}(\varphi(s_0))^3  + |\beta'|_{L^{\infty}(\mathbb{R})}\frac{b^*}{H} \int_0^{t_1}|b_t(t)| dt  \nonumber \\
& -\int_0^{t_1} \frac{k}{2}|u_z(t, s(t))|^2s_t(t) dt \mbox{ for }t_1\in [0, T], 
\end{align}
where $\tilde{u}$ and $\tilde{u}_0$ are the same functions defined by (\ref{p2-0}) and  (\ref{p2-6}). 
Note that the last term of the right-hand side of (\ref{eqb-7}) can be dealt as 
\begin{align}
\label{eqb-8}
& -\int_0^{t_1} \frac{k}{2}|u_z(t, s(t))|^2s_t(t) dt =-\frac{k}{2}\int_0^{t_1}\biggl|\frac{1}{k}(u(t, s(t))s_t(t) \biggr|^2 s_t(t) dt \nonumber \\
& \leq \frac{a_0}{2k} \left(\frac{b^*}{H}\right)^3 \int_0^{t_1}|s_t(t)|^2 dt \leq \frac{a_0}{2k} \left(\frac{b^*}{H}\right)^3 C_1(1+F(t_1)), 
\end{align}
where $C_1$ is the same constant as in (\ref{eqb0-1}). 
Consequently, by recalling the estimate (\ref{eq4-25})-(\ref{eq4-27}) and using (\ref{eqb-8}) we also see that there exists a positive constant $C_2$ such that (\ref{eqb0-2}) holds. 
\end{proof}
\subsection{Proof of Theorem \ref{b-2}}\label{bb}
At the end of the paper, by using the uniform estimate with respect to time obtained in the previous section, we show Theorem \ref{b-2} on the large-time behavior of a solution of (P)$_{u_0, s_0, b}$. 
\newline
Let us assume (A1), (A2)', (A3), (A4)' and (A5). Then, by Theorem \ref{th2}, we have a solution $(s, u)$ of (P)$_{u_0, s_0, b}$ on $[0, T]$ for any $T>0$ such that $\varphi(a)\leq u\leq b^*H^{-1}$ on $[a, s(t)]$ for $t>0$. Also, by Lemma \ref{lem2}, the free boundary $s$ has a lower bound $s_*$ such that $s_*>a$. 
We put $F(t)=\int_0^t\beta(b(\tau)-Hu(\tau, a))d\tau$ for $t>0$ and $r_*=a+H (b^*)^{-1}|u_0|_{L^1(a, s_0)}$. Since $F(t)$ is monotone increasing function with respect to $t$, either of the following cases can be considered: 
\begin{itemize}
\item[(Case1):] $F(t)\leq b^*H^{-1}(r_{\varphi}-r_*)$ for $t>0$;
\item[(Case2):] there exists a positive constant $M>b^*H^{-1}(r_{\varphi}-r_*)$ and $t_0>0$ such that $b^*H^{-1}(r_{\varphi}-r_*)<F(t)\leq M$ for $t\geq t_0$;
\item[(Case3):] for any positive constant $M$, there exists $t_{M}>0$ such that $F(t)>M$ for any $t>t_{M}$. 
\end{itemize}
Now, we show that either of (Case1) or (Case3) holds by contradiction, namely, we assume (Case2). 
Then, we note that $s$ is bounded on $[0, \infty)$. Indeed,  from (\ref{eq3-2}), it holds that 
\begin{align}
\int_a^{s(t)}u(t) dz = \int_a^{s_0}u_0 dz + \int_0^t\beta(b(\tau)-Hu(t, a)) d\tau. \nonumber 
\end{align}
Therefore, by the boundedness of $u$, we see that 
\begin{align}
\label{eqb-43}
s(t) &\leq a+ \frac{1}{\varphi(a)}\biggl[\frac{b^*}{H}(s_0-a) + \int_0^t\beta(b(\tau)-Hu(t, a)) d\tau\biggr] \mbox{ for }t>0. 
\end{align}
Since $F(t)\leq M$ for any $t\geq t_0$, (\ref{eqb-43}) implies that $s$ is bounded on $[0, \infty)$. 
On account of the boundedness of $s$, (\ref{eqb0-1}) and (\ref{eqb0-2}) in Lemma \ref{b-3}, (\ref{eq5-2}) and (\ref{eq5-3}) we see that there exist $C'_1>0$ and $C'_2>0$ which is independent of time $t$ such that 
\begin{align}
\label{eqb-44}
\int_0^t |s_t(\tau)|^2 d\tau  + \int_0^t |\tilde{u}_y(\tau)|^2_{L^2(0, 1)} d\tau \leq C'_1 \mbox{ for }t>0, 
\end{align}
and 
\begin{align}
\label{eqb-45}
\int_{0}^t|\tilde{u}_t(\tau)|^2_{L^2(0, 1)}d\tau + |\tilde{u}_y(t)|^2_{L^2(0, 1)} \leq C'_2 \mbox{ for }t>0. 
\end{align}
Here, for $\{t_n\}$ such that $t_n\to \infty$ as $n\to \infty$, we put $\tilde{u}_n(t, y):=\tilde{u}(t+t_n, y)$ for $(t, y)\in [0, 1]\times [0, 1]$, $s_n(t):=s(t+t_n)$, $b_n(t):=b(t+t_n)$ for $t\in [0, 1]$. By (A2)', it is clear that $b_n \to b_{\infty}$ in $L^1(0, 1)$ as $n\to \infty$. Also, by (\ref{eqb-44}) and (\ref{eqb-45}) we see that $s_{nt}$ is bounded in $L^2(0, 1)$ and $\tilde{u}_n$ is bounded in $W^{1, 2}(0, 1; L^2(0, 1))\cap L^{\infty}(0, 1; H^1(0, 1))$ and can take a subsequence $\{n_j\}\subset \{n\}$ such that the following convergences holds for some $\tilde{u}_{\infty}\in H^1(0, 1)$ and $s_{\infty} \in \mathbb{R}$ satisfying $s_{\infty}>a$:
$$
\begin{cases}
\tilde{u}_{nj}(0) = \tilde{u}(t_{nj}) \to \tilde{u}_{\infty} \mbox{ in }C([0, 1]), \mbox{ weakly in }H^1(0, 1), \mbox{ and } s_{nj}(0)=s(t_{nj}) \to s_{\infty} \mbox{ in } \mathbb{R}, \\
\tilde{u}_{njt} \to 0 \mbox{ in } L^2(0, 1; L^2(0, 1)) \mbox{ and } s_{njt} \to 0 \mbox{ in } L^2(0, 1), \\
\tilde{u}_{nj} \to \tilde{u}_{\infty} \mbox{ in }C([0, 1]; L^2(0, 1)), \\
\hspace{1.8cm} \mbox{ weakly in }W^{1, 2}(0, 1;L^2(0, 1)), \mbox{ weakly -* in }L^{\infty}(0, 1; H^1(0, 1)), \\
s_{nj} \to s_{\infty} \mbox{ in }C([0, 1]), \mbox{ weakly in }W^{1, 2}(0, 1)
\end{cases}
$$
as $j\to \infty$.  The strong convergence of $\tilde{u}_{nj}$ is obtained by the Ascoli-Arzela theorem. Also, by the Sobolev's embedding theorem in one dimension (cf. (\ref{eqb-41})), it holds that for $y\in [0, 1]$, 
\begin{align}
\label{eqb-50}
& \int_0^1|\tilde{u}_n(t, y)-\tilde{u}_{\infty}(y)|^2 dt \leq C_e \int_0^1 |\tilde{u}_n(t)-\tilde{u}_{\infty}|_{H^1(0, 1)}|\tilde{u}_n(t)-\tilde{u}_{\infty}|_{L^2(0, 1)} dt. 
\end{align}
Thereofore, by the strong convergence of $\tilde{u}_{nj}$ and (\ref{eqb-50}) we see that 
\begin{align}
\label{eqb-51}
\tilde{u}_{nj}(y) \to \tilde{u}_{\infty}(y) \mbox{ in } L^2(0, 1) \mbox{ at }y=0, 1 \mbox{ as }j\to \infty. 
\end{align}
Now, 
for each $n$, $(s_{n}, \tilde{u}_{n})$ satisfies that 
$$
\begin{cases}
\tilde{u}_{nt}(t, y)-\frac{k}{(s_n(t)-a)^2}\tilde{u}_{nyy}(t, y)=\frac{ys_{nt}(t)}{s_n(t)-a}\tilde{u}_{ny}(t, y) \mbox{ for }(t, y)\in Q(1), \\
-\frac{k}{s_n(t)-a}\tilde{u}_{ny}(t, 0)=\beta(b_n(t)-H\tilde{u}_n(t, 0)) \mbox{ for }t\in(0, 1), \\
-\frac{k}{s_n(t)-a}\tilde{u}_{ny}(t, 1)=\tilde{u}_n(t, 1)s_{nt}(t) \mbox{ for }t\in (0, 1),  \\
s_{nt}(t)=a_0(\tilde{u}_n(t, 1)-\varphi(s_n(t))) \mbox{ for }t\in (0, 1).
\end{cases}
$$
Using the convergences of $\tilde{u}_{nj}$ and $s_{nj}$ we infer that $\tilde{u}_{\infty} \in H^2(0, 1)$ and 
\begin{align}
\label{eqb-57}
-\frac{k\tilde{u}_{\infty yy}}{(s_{\infty}-a)^2} =0 \mbox{ a.e. on } (0, 1). 
\end{align}
Also, by noting that 
\begin{align}
\label{eqb-58}
& |\beta(b_n(t)-H\tilde{u}_n(t, 0))-\beta(b_{\infty}-H\tilde{u}_{\infty}(0))| \nonumber \\
\leq & |\beta'|_{L^{\infty}(\mathbb{R})} (|b_n(t)-b_{\infty}| + H|\tilde{u}_n(t, 0)-\tilde{u}_{\infty}(0)|), 
\end{align}
we easliy see that $\beta(b_{nj}-H\tilde{u}_{nj}(0)) \to \beta(b_{\infty}-H\tilde{u}_{\infty}(0))$ in $L^1(0, 1)$ as $j\to \infty$ and that 
\begin{align}
\label{eqb-59}
-\frac{1}{s_{\infty}-a}\tilde{u}_{\infty y}(0)=\beta(b_{\infty}-H\tilde{u}_{\infty}(0)), \ -\frac{1}{s_{\infty}-a}\tilde{u}_{\infty y}(1)=0. 
\end{align}
Moreover, the strong convergences of $s_{njt}$, $s_{nj}$ and (\ref{eqb-51}) yield 
\begin{align}
\label{eqb-60}
\tilde{u}_{nj}(\cdot, 1)-\varphi(s_{nj}(\cdot)) \to 0 =\tilde{u}_{\infty}(1)-\varphi(s_{\infty}) \mbox{ in }L^2(0, 1) \mbox{ as }j\to \infty.
\end{align}
As a result, by (\ref{eqb-57})-(\ref{eqb-60}) we see that $(s_{\infty}, \tilde{u}_{\infty})$ is a solution of $(\tilde{\mbox{P}})_{\infty}$ and by Lemma \ref{b-1} $\varphi(s_{\infty})\geq b_{\infty}H^{-1}$ and $\tilde{u}_{\infty}(y)=\varphi(s_{\infty})$ for $y\in (0, 1)$. 
This is a contradiction for $W=\phi$ (see Lemma \ref{b-1} below). Therefore, we conclude either of (Case1) or (Case3).
If (Case3), then this with (\ref{eqb-00}) means (i) of Theorem \ref{b-3}. On the other hand, since $F(t)$ is monotone incleasing function with respect to time $t$, if (Case1) then $\lim_{t \to \infty} F(t)\leq b^*H^{-1}(r_{\varphi}-s_*)$. Also, from (\ref{eq3-2}), we have that 
\begin{align}
s(t) &\leq a+ \frac{1}{\varphi(a)}\biggl[\frac{b^*}{H}(s_0-a) + \int_0^t\beta(b(\tau)-Hu(t, a)) d\tau\biggr] \nonumber \\
& \leq a+ \frac{1}{\varphi(a)}\frac{b^*}{H}\biggl[(s_0-a) +r_{\varphi}-r_*\biggr] \mbox{ for }t>0, \nonumber 
\end{align}
which implies (ii) of Theorem \ref{b-3}. This concludes that Theorem \ref{b-3} holds. 



\subsection*{Acknowledgments}
This research was supported by  Grant-in-Aid No.16K17636, JSPS.

\end{document}